\documentclass[12pt]{amsart}
\pdfoutput=1
\usepackage{amssymb}
\usepackage[latin1]{inputenc}
\usepackage{amsmath}
\usepackage{amsfonts}
\usepackage[usenames]{color}
\usepackage{graphicx}
\usepackage{array}
\usepackage{psfrag}
\usepackage{color}
\usepackage{ulem}
\usepackage{fancyhdr}

\makeatletter
 
 \@addtoreset{equation}{section}
\makeatother

\newtheorem{theorem}{Theorem}

\theoremstyle{definition}

\theoremstyle{remark}
\newtheorem{remark}[theorem]{Remark}

\numberwithin{equation}{section}

\begin{document}

\title{An Infinite Family of Linklessly Embeddable Tutte-4-Connected Graphs }

\date{\today}
\author{
Andrei Pavelescu and
Elena Pavelescu
}

\address{
Department of Mathematics, University of South Alabama, Mobile, AL  36688, USA.
}

\maketitle
\rhead{Tutte-4-Connected}

\begin{abstract}
For each $n\ge 14$, we provide an example of a linklessly embeddable, Tutte-4-connected graph of order $n$.\end{abstract}

\section{Introduction}

In \cite{Ma}, Maharry proved that the graph $Q_{13,3}$, previously studied in \cite{Ra} and \cite{Th}, is 4-connected, triangle free graph (which he called Tutte-4-connected), and linklessly embeddable. 
It was only the second known such example.
Through private communication with Maharry, Robertson had conjectured that $K_{5,5}$ minus a perfect matching was the only graph with these properties, so Maharry's discovery was quite extraordinary.
 In this article we prove that the class of linklessly embeddable Tutte-4-connected graphs is quite rich, as we provide an example of such a graph for every order $n\ge 14$.

All the graphs considered in this article are finite and simple. A graph is called {\bf intrinsically linked} (IL) if every embedding of it in the three dimensional space contains a nontrivial link. 
A graph which is not IL is called {\bf linklessly embeddable} (nIL). 
The class of nIL graphs is closed under taking minors.
By work of Conway and Gordon \cite{CG}, Robsertson, Seymour, and Thomas \cite{RST}, and Sachs \cite{Sa}, a graph is nIL if and only if it does not contain any of the graphs in the Petersen family as a minor. 
The Petersen family consists of  the seven graphs obtained from $K_6$ by performing all the possible $\nabla Y-$moves and $Y\nabla-$moves, as described in Figure~\ref{fig-ty}.


\begin{figure}[htpb!]
\begin{center}
\begin{picture}(160, 50)
\put(0,0){\includegraphics[width=2.4in]{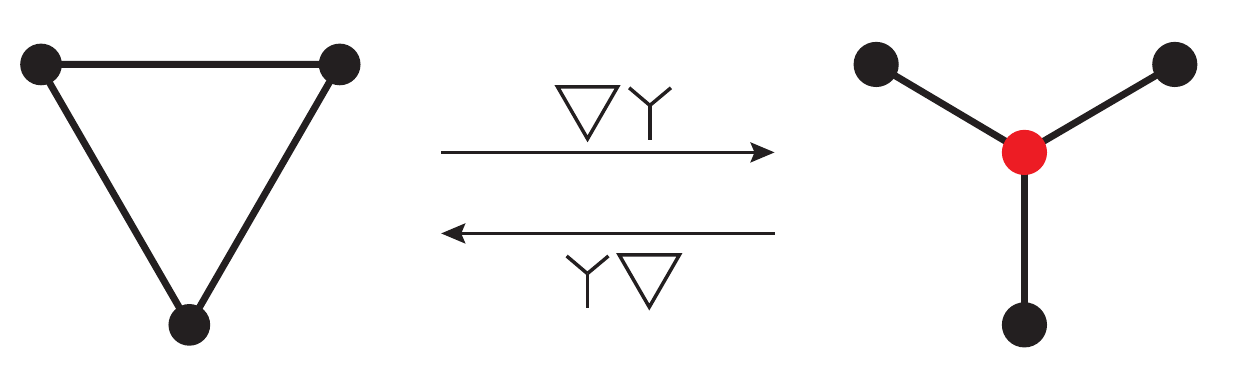}}
\end{picture}
\caption{\small  $\nabla Y-$ and $Y\nabla-$moves} 
\label{fig-ty}
\end{center}
\end{figure}

A graph is called {\bf apex} if it has a vertex whose deletion gives a planar graph. 
An apex graph is nIL \cite{Sa}.

A graph $G$ is the {\bf clique sum} of $G_1$ and $G_2$ over $K_p$ if $V(G)=V(G_1)\cup V(G_2)$, $E(G)=E(G_1)\cup E(G_2)$ and the subgraphs induced by $V(G_1)\cap V(G_2)$ in $G_1$ and $G_2$ are   both complete of order $p$.
 We use the notation $G=G_1\oplus_{K_p}G_2$.  
\section{Main Theorem}

In his 1974 paper \cite{Sl} classifying 4-connected graphs, Slater introduced the notion of $n$-vertex splitting. 
Given a simple graph $G$ and a vertex $v$ of $G$ of degree at least $2n-2$, a  {\bf $n$-vertex splitting}  is obtained by replacing $v$ by two adjacent vertices $a$ and $b$ and by adding one edge from each former neighbor of $v$ to either $a$ or $b$, but not both, such that the degree of both $a$ and $b$ is at least $n$.
 Theorems 0 and 1 of \cite{Sl} state that any graph obtained from an $n$-connected graph by either adding edges or performing $n$-vertex splittings remains $n$-connected. 
 We repeatedly use these results in our constructions.

To simplify arguments about linkless embedabbility we use a result of Holt, Lov\'asz, and Schrijver \cite{HLS}.
 It states that if $G$ is the clique sum over $S$ of two nIL graphs, then $G$ is IL if and only if one can contract two or three components of $G-S$ so that the contracted nodes together with $S$ form a $K_7$ minus a triangle.
 In particular, the clique sum over $K_4$ of two connected nIL graphs, with the extra condition that the clique is not a vertex cut in either of the two graphs, is a nIL graph.

Consider the graph $T$ of Figure \ref{T}(a). Deleting any of its vertices produces the planar graph $T'$ in Figure \ref{T}(b), thus $T$ is an apex graph and therefore it is linklessly embeddable.

\begin{figure}[h]
\begin{tabular}{ll}
\includegraphics[scale=.9]{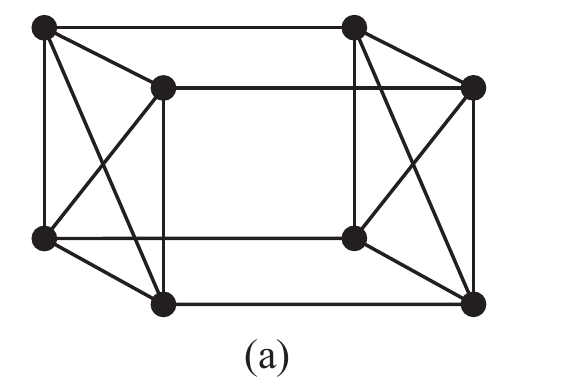}
&
\includegraphics[scale=.9]{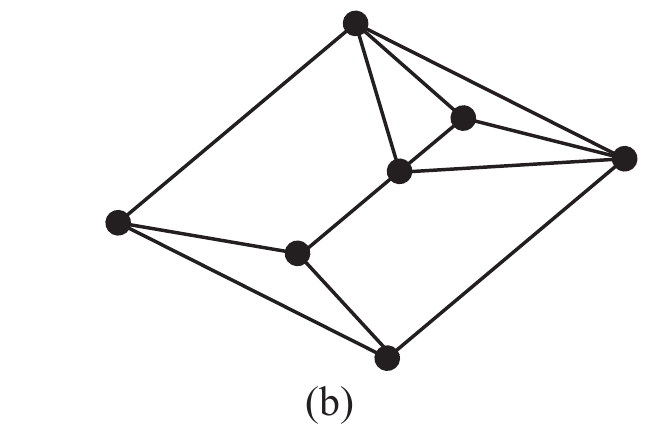}
\end{tabular}
\caption{(a) The graph $T$ is apex.\,\,\,\,\,\,\,\,\,(b) The planar graph $T'$.}
\label{T}
\end{figure}

\noindent Also consider the graph $C$ of Figure \ref{C}(a). Deleting the vertex labeled $a$, produces the planar graph $C'$ in Figure \ref{C}(b),  thus $C$ is also an apex graph and therefore linklessly embeddable.
\begin{figure}[h]
\begin{tabular}{ll}
\includegraphics[scale=1.1]{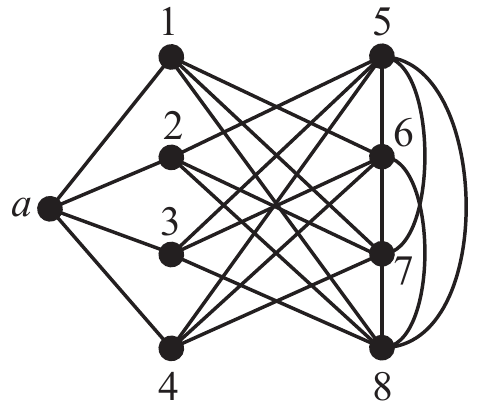} 
&
\includegraphics[scale=1.1]{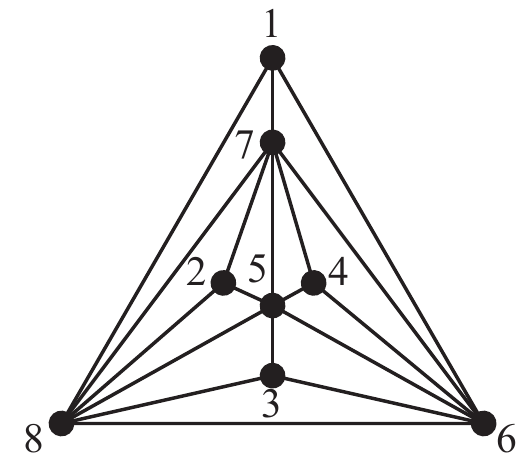}
\end{tabular}
\caption{(a) The graph $C$ is apex.\,\,\,\,\,\,\,\,\,(b) The graph $C'$ is planar.}
\label{C}
\end{figure}
We can now state the main result of this article.

\begin{theorem} For every integer $n\ge 13$, there exists a linklessly embeddable Tutte-4-connected graph of order $n$, $T_n$.
\label{main}
\end{theorem}

\begin{proof} Take $T_{13}=Q_{13,3}$, the graph thoroughly analyzed in \cite{Ma} and displayed in Figure \ref{maharry}.
\begin{figure}[htpb!]
\begin{center}
\begin{picture}(140, 120)
\put(0,0){\includegraphics[width=1.8in]{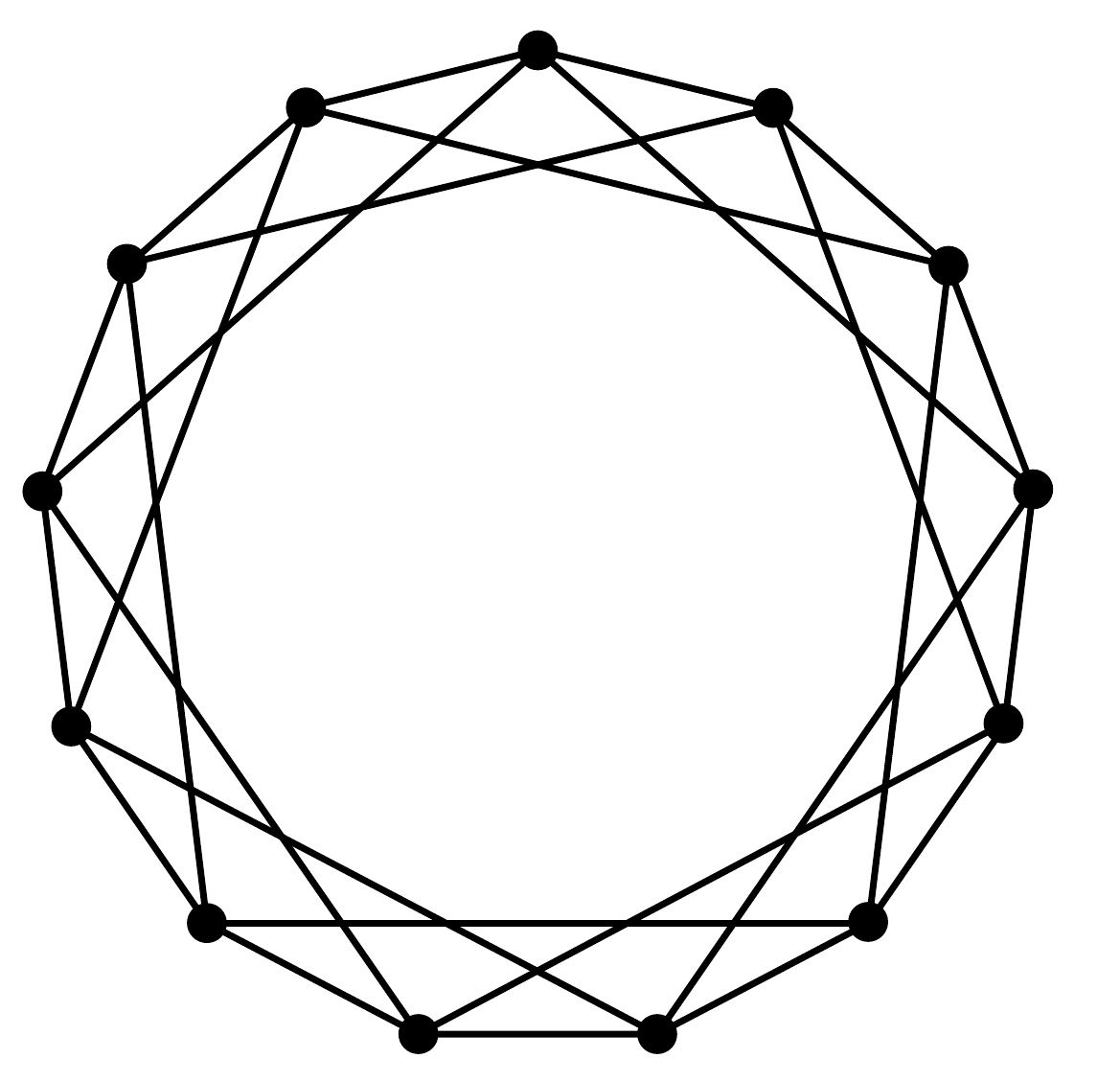}}
\end{picture}
\caption{$Q_{13,3}$ is 4-connected, triangle free, and maximal linklessly embeddable.}
\label{maharry}
\end{center}
\end{figure} 
The graph $T_{14}$ is constructed from $T_{10}$ (see Figure \ref{T10})  through successive 4-vertex splittings and edge additions. 
\begin{figure}[htpb!]
\begin{center}
\begin{picture}(150, 110)
\put(0,0){\includegraphics[width=2.3in]{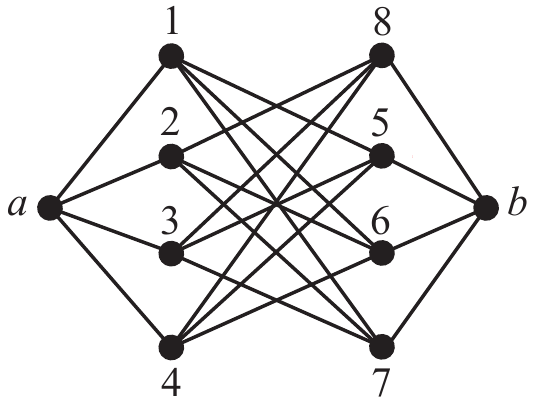}}
\end{picture}
\caption{$T_{10}$ is $K_{5,5}$ minus a perfect matching, the original Tutte-4-connected graph.}
\label{T10}
\end{center}
\end{figure} 
Start by adding the 6 missing edges in the subgraph induced by $\{5,6,7,8\}$ to form an induced $K_4$. 
We call this graph $S_{10}$. 
See Figure \ref{S10}.
This graph is no longer triangle free, but it is still 4-connected and nIL, as it is the clique sum over $K_4$ of $C$ and $K_5$. 
\begin{figure}[htpb!]
\begin{center}
\begin{picture}(150, 110)
\put(0,0){\includegraphics[width=2.3in]{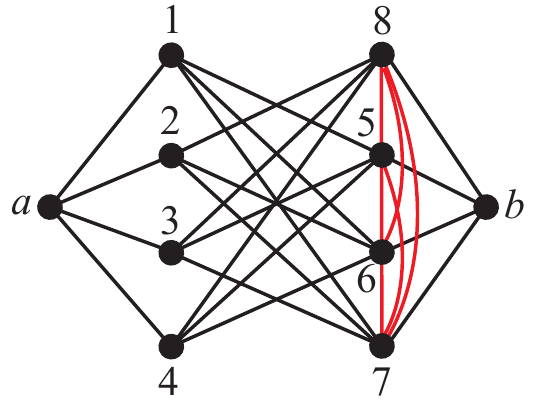}}
\end{picture}
\caption{$S_{10}$ is isomorphic to $C\oplus_{K_4}K_5$.}
\label{S10}
\end{center}
\end{figure} 
Then perform a sequence of vertex splittings and edge additions as detailed in Figure \ref{S11-14}.
Build $S_{11}$ from $S_{10}$ by splitting the vertex labeled 8 into $\{8,9\}$ and add the edge $\{6,9\}$. 
To build $S_{12}$ from $S_{11}$, split the vertex labeled 5 into $\{5,10\}$ and add the edge $\{7,10\}$. 
The graph $S_{13}$ is obtained by splitting the vertex labeled 6 into $\{6,11\}$ and $T_{14}$ is obtained by further splitting the vertex labeled $7$ into $\{7,12\}$.
\begin{figure}[htpb!]
\begin{center}
\begin{picture}(400, 260)
\put(0,0){\includegraphics[width=5.2in]{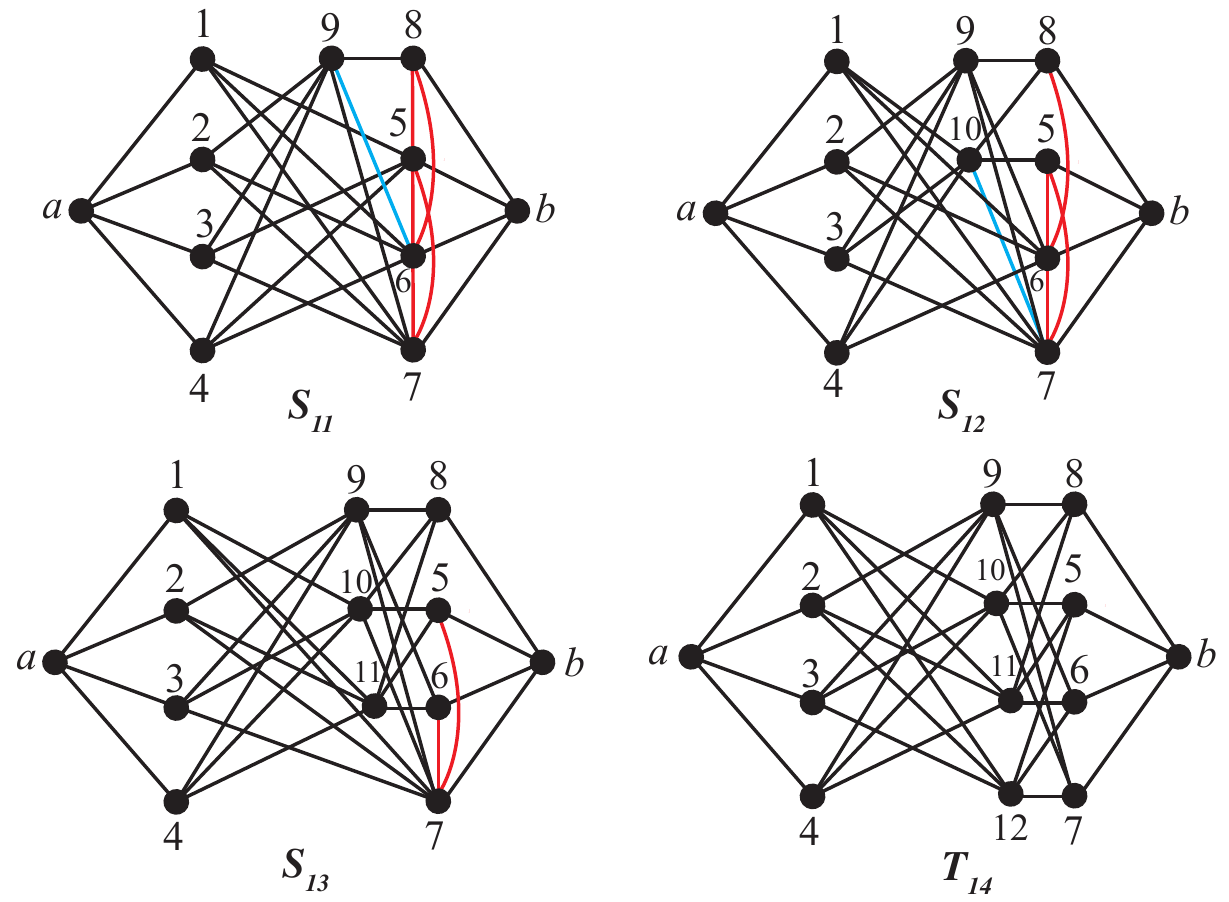}}
\end{picture}
\caption{Building $T_{14}$ from $T_{10}$.}
\label{S11-14}
\end{center}
\end{figure} 
 The graph $T_{14}$, redrawn for clarity in Figure \ref{T14}(a), was obtained from a 4-connected graph by 4-vertex splittings and edge additions, so it is 4-connected, by Theorems 0 and 1 of \cite{Sl}. It is also a triangle-free graph. Finally, $T_{14}$ is a subgraph of $S_{14}$ of Figure \ref{T14}(b), which is a nIL graph as $S_{14}\simeq C\oplus_{K_4} C$.
 Thus $T_{14}$ is also nIL.
\begin{figure}[h]
\begin{tabular}{ll}
\includegraphics[scale=0.85]{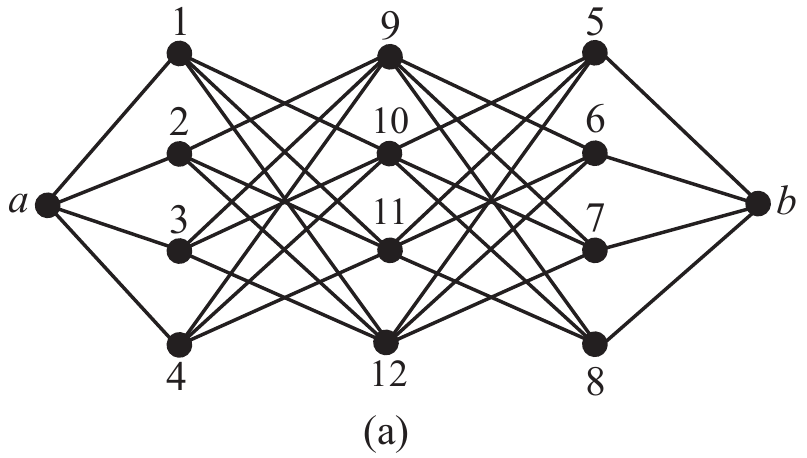}
&
\includegraphics[scale=0.85]{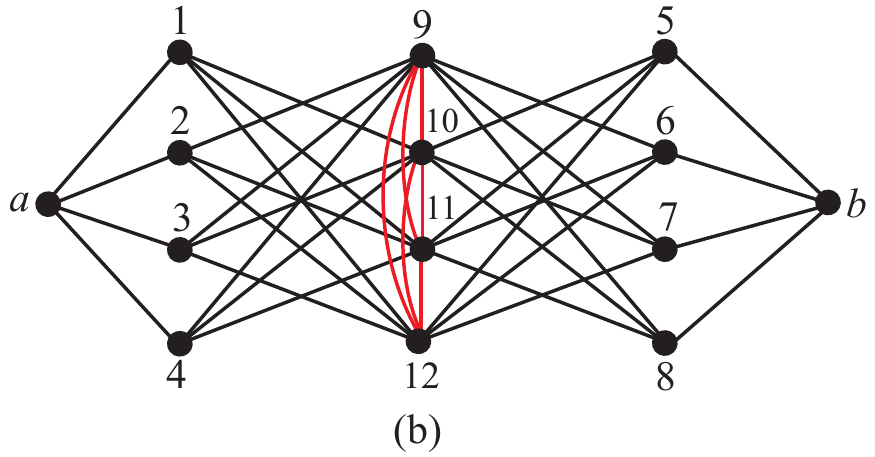}
\end{tabular}
\caption{(a) The graph $T_{14}$.\,\,\,\,\,\,\,\,\,(b) The graph $S_{14}$.}
\label{T14}
\end{figure}

The graphs $T_{15}$, $T_{16}$, $T_{17}$, and $T_{18}$, featured in Figure \ref{T15-18}, are obtained by successive 4-vertex splittings starting with the graph $T_{14}$ and thus they are all 4-connected. 
\begin{figure}[htpb!]
\begin{center}
\begin{picture}(420, 230)
\put(0,0){\includegraphics[width=5.7in]{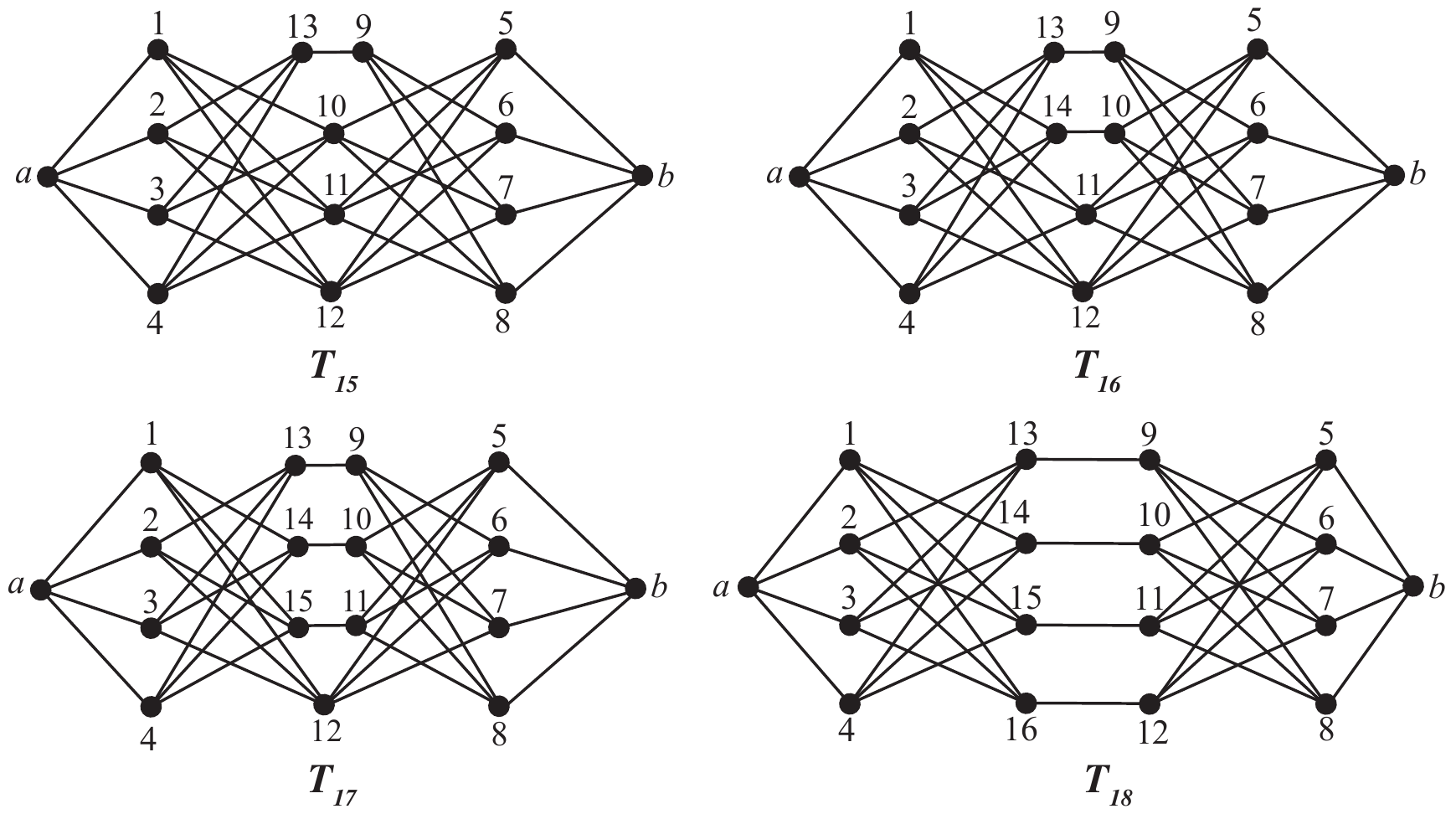}}
\end{picture}
\caption{Graphs $T_{15}$ through $T_{18}$.}
\label{T15-18}
\end{center}
\end{figure} 
Since vertex splittings do not create induced triangles, all four graphs are triangle-free. 
Finally, they are all minors of the graph $S_{18}$ featured in Figure \ref{S18} which is nIL as the consecutive clique sums over $K_4$ of nIL graphs ($S_{18}\simeq C \oplus_{K_4}T \oplus_{K_4} C $). 
This implies that $T_{15}$, $T_{16}$, $T_{17}$, and $T_{18}$ are all nIL. \\

\begin{figure}[htpb!]
\begin{center}
\begin{picture}(300, 115)
\put(0,0){\includegraphics[width=4in]{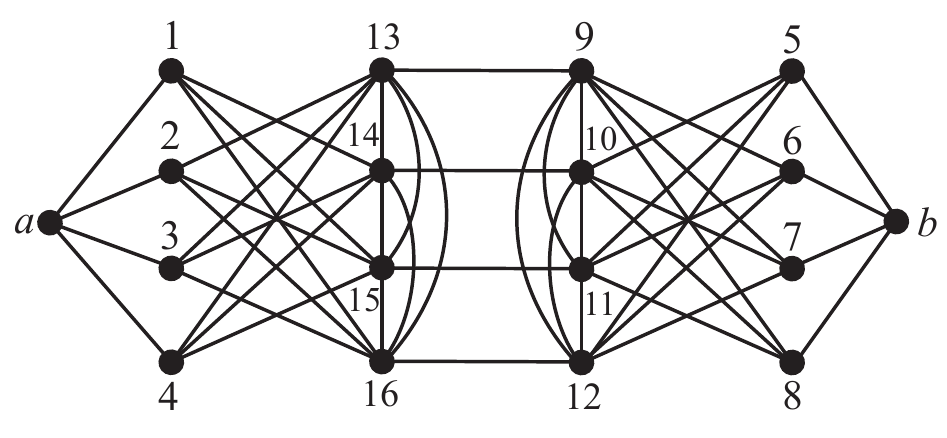}}
\end{picture}
\caption{The graph $S_{18}$ is isomorphic to $C\oplus_{K_4}T \oplus_{K_4}C$.}
\label{S18}
\end{center}
\end{figure} 

The next four graphs in this sequence are obtained from $T_{18}$ by edge additions and 4-vertex splittings. The choices of edges to be added and vertices to be split generalize to higher order graphs in this sequence and constitute the inductive step for the construction of all the graphs of order more than 22. For reference, we include them in Figure \ref{T19-22}.

\begin{figure}[htpb!]
\begin{center}
\begin{picture}(420, 210)
\put(0,0){\includegraphics[width=5.8in]{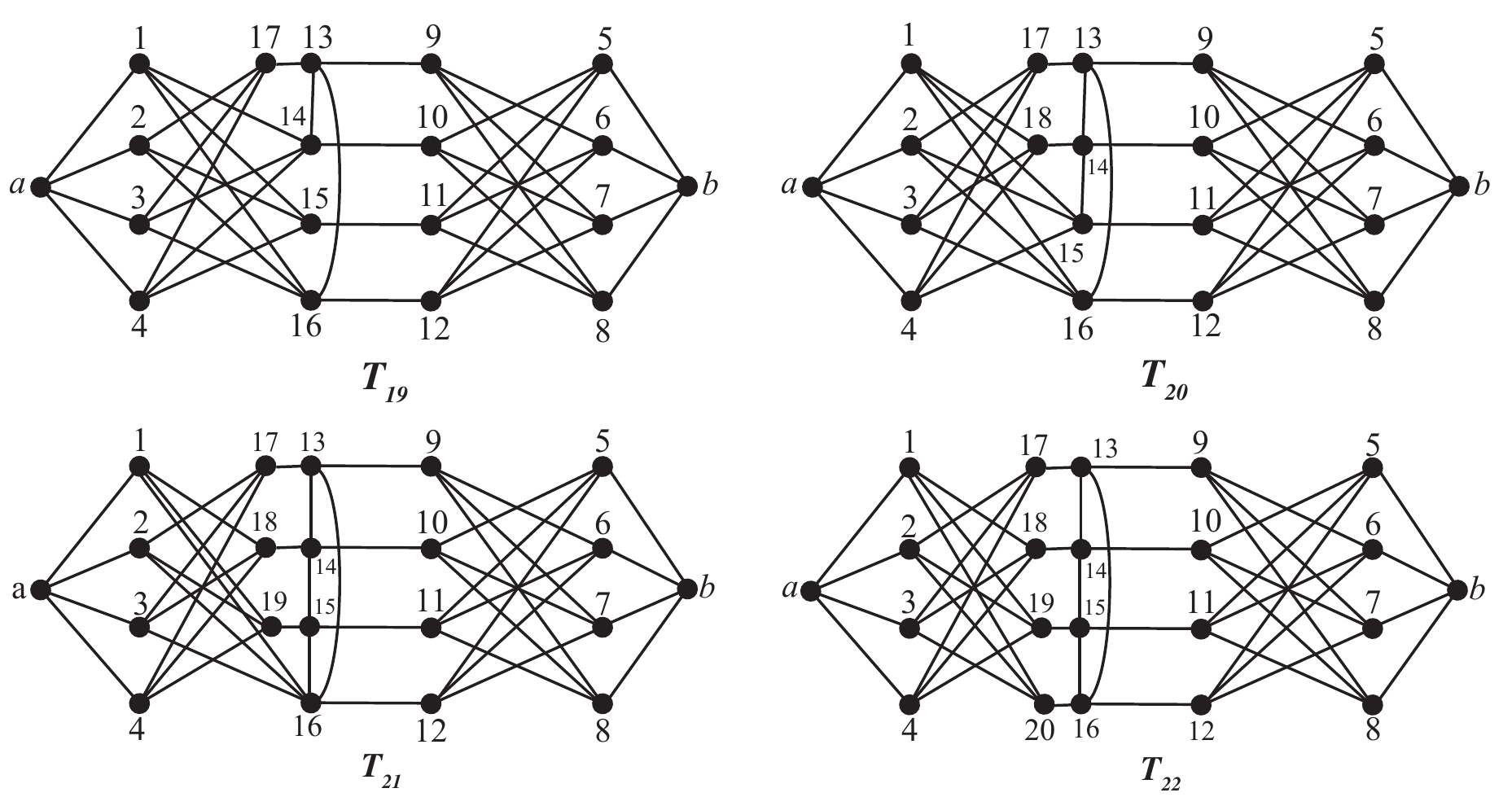}}
\end{picture}
\caption{The graphs $T_{19}$ through $T_{22}$.}
\label{T19-22}
\end{center}
\end{figure} 

The graph $T_{19}$ is obtained from $T_{18}$ by adding the edges $\{13,14\}$ and $\{13,16\}$, so that the degree of the vertex labeled 13 becomes 6, and then by 4-splitting this vertex into $\{13,17\}$. The graph $T_{20}$ is obtained from $T_{19}$ by first adding the edge $\{14,15\}$ and then by 4-splitting the vertex labeled 14 into $\{14,18\}$. The graph $T_{21}$ is obtained from $T_{20}$ by first adding the edge $\{15,16\}$ and then by 4-splitting the vertex labeled 15 into $\{15,19\}$. Finally, the graph $T_{22}$ is obtained by executing a 4-vertex splitting on the vertex labeled 16 into $\{16,20\}$. 

Since each of these four graphs was obtained by edge additions and 4-vertex splittings starting from a 4-connected graph, they are all 4-connected, by Theorems 0 and 1 of \cite{Sl}. One can easily check the graphs are triangle-free. Since they are all minors of the graph $S_{22}\simeq C\oplus_{K_4}T \oplus_{K_4}T \oplus_{K_4} C$, featured in Figure \ref{S22}, which is nIL, all the four graphs must be nIL.

\begin{figure}[htpb!]
\begin{center}
\begin{picture}(340, 120)
\put(0,0){\includegraphics[width=4.5in]{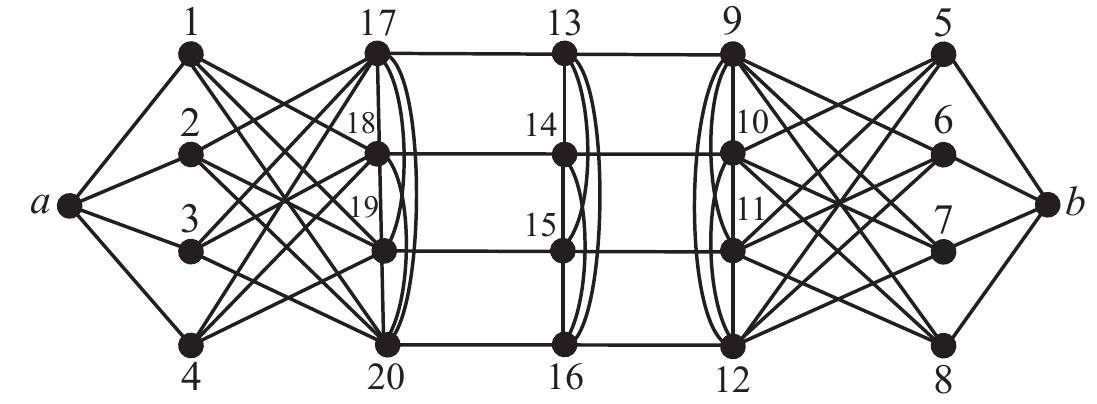}}
\end{picture}
\caption{The graph $S_{22}$ is isomorphic to $C\oplus_{K_4}T \oplus_{K_4}T \oplus_{K_4}C$.}
\label{S22}
\end{center}
\end{figure}

\begin{figure}[htpb!]
\begin{center}
\begin{picture}(320, 120)
\put(0,0){\includegraphics[width=4.5in]{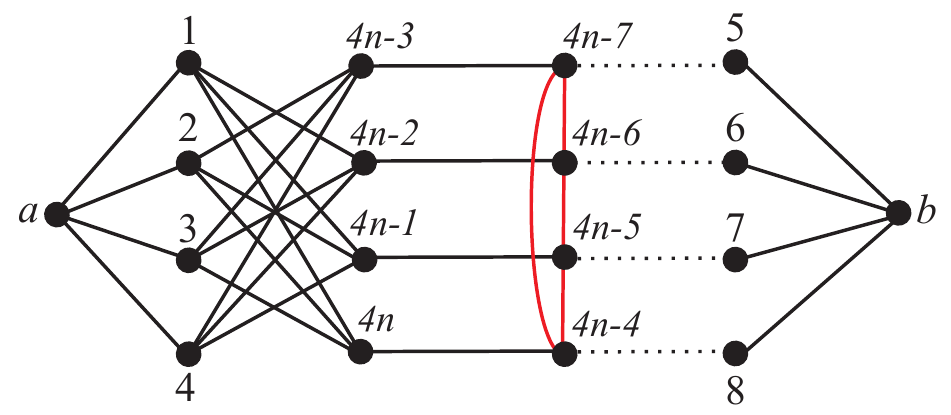}}
\end{picture}
\caption{The graph $T_{4n+2}$.}
\label{T4n+2}
\end{center}
\end{figure}
\end{proof}
In general, for $n\ge 5$, assuming the nIL, Tutte-4-connected graph $T_{4n+2}$ has already been build (see Figure \ref{T4n+2}), by the method described above one can build, by successive edge additions and 4-vertex splittings at  vertices labeled $4n-3$, $4n-2$, $4n-1$ , and $4n$, the graphs $T_{4n+3}, T_{4n+4}, T_{4n+5}, $ and $T_{4n+6}$.
These graphs are 4-connected by Theorems 0 and 1 of \cite{Sl}, and triangle-free by construction. They are all minors of the graph $S_{4n+6}$, obtained by adding all the missing edges in the subgraphs of $T_{4n+6}$
 induced by the vertices labeled $\{4k+1, 4k+2, 4k+3, 4k+4\}$, for $2\le k \le n$. This graph is isomorphic to $C \oplus_{K_4}T\oplus_{K_4}\ldots \oplus_{K_4}T\oplus_{K_4}C$, where the clique sum contains $(n-2)$ $T$-summands. The graph $S_{4n+6}$ is therefore nIL, and so must be $T_{4n+3}$ through $T_{4n+6}$.

\begin{remark} In this article, we've presented an example of a nIL Tutte-4-connected graph for every order $n\ge 14$. The constructions are dependent on the choices of edges to be added and vertices to be 4-vertex split from one order to another. For instance, the graphs of order $4n+2$, for $n>4$, contain the  4-cycles $\{4k+1, 4k+2, 4k+3, 4k+4\}$, for $2\le k \le n$. For each index $2\le k \le n$, one could pick another 4-cycle on the vertices $\{4k+1, 4k+2, 4k+3, 4k+4\}$ and still obtain a nIL, Tutte-4-connected graph. 
\end{remark}

\begin{remark} The graph $Q_{13,3}$, analyzed by Maharry in \cite{Ma}, is a splitter in the class of nIL graphs. In particular, none of the graphs we constructed in this article contains $Q_{13,3}$ as a minor. It would be extremely interesting to know whether there is another example of a Tutte-4-connected graph  of order larger than 13, which is a splitter for the class of nIL graphs.
\end{remark}

\bibliographystyle{amsplain}

\end{document}